\documentclass{scrartcl}
\usepackage[british]{babel}
\usepackage[textwidth=150mm,top=2cm]{geometry}

\usepackage{dsfont, amsthm, mathtools,  enumerate,amsmath,amsrefs,array}

\usepackage{dsfont}
\usepackage{amssymb}
\usepackage{aurical}
\usepackage{mathtools}
\usepackage{longtable}
\usepackage{wasysym}
\usepackage{euscript}
\usepackage{mathrsfs}
\usepackage{tikz}
\usepackage[hidelinks]{hyperref}

\usepackage{libertine}

\title{On $\alpha$-minimizing hypercones}

\author{Peter \textsc{Lewintan}\footnote{peter.lewintan@uni-due.de, University of Duisburg-Essen, Germany}}
\date{October 29, 2018}

\theoremstyle{plain}
\newtheorem{theorem}{Theorem}[section]

\newtheorem{lemma}[theorem]{Lemma}
\newtheorem{proposition}[theorem]{Proposition}

\newtheorem*{acknowledgement}{Acknowledgement}

\theoremstyle{definition}
\newtheorem{definition}[theorem]{Definition}

\theoremstyle{remark}
\newtheorem{remark} [theorem]{Remark}

\newtheorem{example}[theorem]{Example}

\usepackage{dsfont,mathrsfs,euscript,nicefrac,multirow,arydshln}
\newcommand{\boldm}[1] {\mathversion{bold}#1\mathversion{normal}}
\newcommand{\R}{\mathds R}
\newcommand{\p}{\text{\boldm{$p_{m}$}}}
\newcommand{\Hdif}{H_{m,\alpha}}
\newcommand{\g}{{g_{m,\alpha}}}
\newcommand{\gam}{\text{\boldm{$\gamma_{m,\alpha}$}}}
\newcommand{\Per}{\EuScript{P}_\alpha}
\newcommand{\neghphantom}[1]{\settowidth{\dimen0}{#1}\hspace*{-\dimen0}}
\renewcommand{\P}{\text{\boldm{$P_{m,\alpha}$}}}
\renewcommand{\t}{\text{\boldm{$t_{m,\alpha}$}}}
\newcommand{\w}{\text{\boldm{$w_{m,\alpha}$}}}
\renewcommand{\v}{\text{\boldm{$v_{m,\alpha}$}}}
\newcommand{\F}{\text{\boldm{$F_{m,\alpha}$}}}

\usepackage{letltxmacro}
\LetLtxMacro{\OldSqrt}{\sqrt}
\newcommand{\ClosedSqrt}[1][\hphantom{3}]{\def\DHLindex{#1}\mathpalette\DHLhksqrt}
\makeatletter
    \newcommand*\bold@name{bold}
    \def\DHLhksqrt#1#2{%
        \setbox0=\hbox{$#1\OldSqrt{#2\,}$}\dimen0=\ht0\relax%
        \advance\dimen0-0.2\ht0\relax
        \setbox2=\hbox{\vrule height\ht0 depth -\dimen0}%
        {\hbox{$#1\expandafter\OldSqrt\expandafter[\DHLindex]{#2\,}$}
        \lower\ifx\math@version\bold@name0.6pt\else0.4pt\fi\box2}
    }
    \renewcommand*{\sqrt}[2][\ ]{\ClosedSqrt[\leftroot{-2}\uproot{1}#1]{#2}\kern0.1em}
\makeatother

\newcommand{\skalarProd}[2]{\big\langle#1,#2\big\rangle}

\newcolumntype{"}{@{\hskip\tabcolsep\vrule width 1pt\hskip\tabcolsep}}

\begin{document}

\maketitle

\begin{abstract}
 In this paper we considerably extend the class of known $\alpha$-minimizing hypercones using sub-calibration methods. Indeed, the improvement of previous results follows from a careful analysis of special cubic and quartic polynomials.
\end{abstract}

\textit{AMS 2010 MSC:}
 Primary: 53C38; Secondary: 53A10, 49Q10, 49Q15, 58E15.

\textit{Keywords:}
 $\alpha$-minimizing hypercones, calibrations, foliations.


\section{Introduction}
Let $P_0$ and $P_1$ be two distinct points in $\R\times\R_{\geq0}$ and consider for $\alpha\geq0$ the variational problem
\[
 \int y^\alpha\ d\mathscr{H}^1(x,y) \ \to \ \min
\]
within the class
\[
 \mathscr{C}\coloneqq\{\, \mathfrak{K}\colon [0,1]\to\R\times\R_{\geq0} \ \text{ Lipschitz s.t. } \ \mathfrak{K}(0)=P_0, \mathfrak{K}(1)=P_1 \,\}.
\]
Hence, with $\alpha=0$ we are looking for the shortest curve joining $P_0$ and $P_1$, with $\alpha=\frac12$ we gain a parametric version of the brachistochrone-problem, and the case $\alpha=1$ leads to rotationally symmetric minimal surfaces in $\R^3$. On the other hand, the variational integral with $\alpha=1$ appears when considering the potential energy of heavy chains.

Of course, the shortest path between $P_0$ and $P_1$ is a line, and the minimizing curve in the case $\alpha=\frac12$ was named brachistochrone. However, the variational problem with $\alpha=1$ may possess two distinct minimizers, namely a catenary and a Goldschmidt curve, which consists of three straight lines, cf. \cite[ch. 8 sec. 4.3]{GH2}.

In order to prove the minimality of the above mentioned curves it is sufficient to embed the corresponding curve into a field of extremals\footnote{An argument which goes back to \textsc{Weierstrass}.}, i.e. into a foliation of extremal curves, cf. \cite[ch. 6 sec. 2.3]{GH1}. In fact, this can be directly justified by the divergence theorem. For this purpose let us look at the vector field
\[
 \xi(x,y)\coloneqq y^\alpha\cdot\nu(x,y),
\]
where $\nu(x,y)$ are the normal fields orienting the curves from the foliation. Since all these curves are extremals, the vector field $\xi$ is divergence-free. The conclusion then follows by applying the divergence theorem to the vector field $\xi$ on the open set which is bounded by a critical curve and a comparison curve. In geometric measure theory setting, the critical curve is said being calibrated by $\xi$, and the vector field $\xi$ is called \textit{calibration}.\footnote{Such method of conclusion is applicable even in a more general context and is well-known as Federer's differential form argument, cf. \cite[5.4.19]{Federer}.}

In this paper we consider the higher dimensional variational problem and prove the minimizing property of special hypercones. Therefor we will construct suitable foliations. The crux hereby is to find an auxiliary function whose level sets are extremals.

First, we will weaken our considerations and look at ``inner'' and ``outer'' variations separately as in \cite{dPP}. This gives simplified proofs and yields sub-solutions and sub-calibrations. The advantage of this weakened ansatz is that we can gain specific auxiliary functions. Moreover, we will show that a careful analysis of extremals as in \cite{Davini} provides better results to our variational problem but loses the concrete representation of an auxiliary function.
\subsection{The main result}
 Let $m\in\{\,2,3,\dots\,\}$ and let $\EuScript{M}$ be an oriented Lipschitz-hypersurface in $\R^m\times\R_{\geq 0}$. Its \emph{$\alpha$-energy} is given by
\begin{equation}\label{eq:Lew:varprob}
 \EuScript{E}_\alpha(\EuScript{M})\coloneqq \int_\EuScript{M} y^\alpha d\mathscr{H}^m(z),
\end{equation}
where we use the notation $z\coloneqq(x,y)\in\R^m\times\R_{\geq0}$ and denote by $\mathscr{H}^m$ the $m$-dimensional Hausdorff measure. We show
\begin{theorem}\label{Satz:Lew:1}
There exists an algebraic number $\alpha_m >\frac2m$ such that the cone
 \[
  \mathcal{C}^\alpha_m\coloneqq \left\{0\leq y\leq \sqrt{\frac{\alpha}{m-1}}\cdot |x|\right\}, \qquad \text{with arbitrary $\alpha\geq \alpha_m$,}
 \]
 is a local $\alpha$-perimeter minimizer in $\R^m\times\R_{\geq0}$.
\end{theorem}

\begin{remark}
 For $\alpha$ an integer, our result is equivalent to the area-minimizing property of the corresponding rotated cones in $\R^{m+\alpha+1}$. Indeed, with our lower bounds presented in rem. \ref{rem:bounds} we recover the area-minimizing property of all Lawson's cones, i.e. of the cones
\[
 C_{k,h}\coloneqq\{\,(x,y)\in\R^k\times\R^h\mid(h-1)|x|^2=(k-1)|y|^2\,\}
\]
with \ $k,h\geq2$ and $k+h\geq9$ \ or \ $(k,h)\in\{(3,5),(5,3),(4,4)\}$, \ cf. \cite{BdGG, Lawson,Simoes}, where $k$ and $h$ take over the parts of $m$ and $\alpha+1$. For further reading on area-minimizing cones, see also \cite{Lawlor} and the references contained therein.
\end{remark}

\begin{remark}
Following the minimal surfaces theory we will introduce the terminology of a \emph{local $\alpha$-perimeter minimizer} in the next section. Alternatively, we could say in theorem \ref{Satz:Lew:1} that the hypercone
\[
\mathcal{M}^\alpha_m\coloneqq\partial \mathcal{C}^\alpha_m=\{\sqrt{m-1}\cdot y=\sqrt{\alpha}\cdot |x|\}, \qquad \text{with arbitrary $\alpha\geq \alpha_m$,}
\]
is $\alpha$-minimizing in $\R^m\times\R_{\geq0}$, where the boundary of $\mathcal{C}^\alpha_m$ is seen with respect to the induced topology.
\end{remark}

\begin{remark}
 In our proof, we will specify polynomials $\p$ which characterize the corresponding $\alpha_m$ as the unique positive root. Moreover, we show $\alpha_m<\frac{12}{m}$, thus $\alpha_m\to 0$ with $m\to\infty$.
\end{remark}

\begin{remark}\label{rem:bounds}

First ( integer ) bounds can be found in \cite{Dierkes:erstErg}, namely
\[
\alpha_2 = 11, ~ \alpha_3=6, ~ \alpha_4=\alpha_5=\alpha_6=3, ~ \alpha_7=\dots=\alpha_{11}=2, ~ \alpha_m=1 \text{ for $m\geq12$}.
\]
Shortly thereafter, they were corrected in \cite{Dierkes:verbessErg} to
\[
 \alpha_2=6, \quad \alpha_3=4, \quad \alpha_4=3, \quad \alpha_5=\alpha_6=2, \quad \alpha_m=1 \quad \text{for ~ $m\geq7$}.
\]
Our investigations show, that they can be improved to

\begin{center}
\begin{tabular}{r@{\ $\approx$\ }r@{.}l}
$\alpha_2$ & 5 & 881525129\\
$\alpha_3$ & 3 & 958758640\\
$\alpha_4$ & 2 & 829350458\\
$\alpha_5$ & 1 & 969224627\\
$\alpha_6$ & 1 & 352500103
\end{tabular}\quad
\begin{tabular}{r@{\ $\approx$\ }r@{.}l}
$\alpha_7$ & 0 & 963594772\\
$\alpha_8$ & 0 & 728989161\\
$\alpha_9$ & 0 & 581153278\\
$\alpha_{10}$ & 0 & 481712568\\
$\alpha_{11}$ & 0 & 410855526
\end{tabular}\quad
\begin{tabular}{r@{\ $\approx$\ }r@{.}l}
$\alpha_{12}$ & 0 & 357996307\\
$\alpha_{13}$ & 0 & 317117533\\
\multicolumn{3}{c}{\ldots}\\
$\alpha_{2017}$ & 0 & 001377480\\
\multicolumn{3}{c}{\ldots}
\end{tabular}
\end{center}
\end{remark}

\begin{remark}
For all $m=2,3,\dots$ we have $m+\alpha_m\geq4+\sqrt{8}$, cf. remark \ref{Lew:Bem:Koeff}, so, direct calculations yield that all hypercones $\mathcal{M}^\alpha_m$, with $\alpha\geq \alpha_m$, are (of course) $\EuScript{E}_\alpha$-stable, see also \cite[p. 168]{DHT}.
\end{remark}

\begin{remark}
Although $\mathcal{M}^5_2$ is $\EuScript{E}_5$-stable, the corresponding cone $\mathcal{C}^5_2$ is not a (local) $5$-perimeter minimizer in $\R^2\times\R_{\geq0}$. Similarly, the hypercone  $\mathcal{M}^1_6$ is $\EuScript{E}_1$-stable, but the cone $\mathcal{C}^1_6$ does not minimize the $1$-perimeter in $\R^6\times\R_{\geq0}$, cf. \cite{Dierkes:verbessErg}. Hence, the optimality question of our $\alpha_m$'s still remains open.
\end{remark}

\section{Notations and preliminary results}
Let $\Omega\subseteq \R^m\times\R_{\geq0}$ be open (with respect to the induced topology) and let $\alpha>0$. We say that $f\in BV^\alpha(\Omega)$\ if $f\in L_1(\Omega)$ and the quantity
\[
\int_\Omega y^\alpha|Df|\coloneqq\sup\left\{\int_\Omega f(z)\operatorname{div}(\psi(z))dz : \psi\in C^1_c(\Omega,\R^{m+1}), |\psi(z)|\leq y^\alpha\right\}
\]
is finite. For a Lebesgue measurable set $E\subseteq\R^m\times\R_{\geq0}$ we call
\[
 \Per(E;\Omega)\coloneqq \int_\Omega y^\alpha|D\chi_E|
\]
the \emph{$\alpha$-perimeter of $E$ in $\Omega$}. Furthermore, we call $E$ an \emph{$\alpha$-Caccioppoli set in $\Omega$} if $E$ has a locally finite $\alpha$-perimeter in $\Omega$, i.e. $\chi_E\in BV^\alpha_{loc}(\Omega)$.

\begin{example}
 By the divergence theorem, if $E\subseteq\R^m\times\R_{\geq0}$ is an open set with regular boundary, then
 \[
  \Per(E;\Omega) = \EuScript{E}_{\alpha}(\partial E\cap\Omega)
 \]
for all open sets $\Omega$.
\end{example}

\begin{remark}
 Of course, several properties of the $\alpha$-perimeter can be directly transferred from the known properties of the perimeter, cf. \cite{Giusti,Maggi}.
\end{remark}

\begin{remark}
 Note that there are $\alpha$-Caccioppoli sets which are not Caccioppoli, i.e. do not possess a locally finite perimeter: In an arbitrary neighborhood of the origin consider the set
 \[
  A\coloneqq\bigcup_{n=0}^\infty A_n,
 \]
 where $A_n$ is a triangle with vertices
 \[\textstyle
 \left(\frac{1}{2^{n+1}},0\right),  ~ \left(\frac{1}{2^{n}},0\right) ~ \text{and} ~ \left(\frac{3}{2^{n+2}},\sqrt{\frac{1}{4(n+1)^2}-\frac{1}{2^{2n+4}}}\right).
 \]
 Hereby, the $A_n$ are chosen in such a way that
 \[
 \left|\partial A_n\cap\big(\R\times\R_{>0}\big)\right| = \frac{1}{n+1}.
 \]
 On the other hand, the $\alpha$-perimeter of $A$ is dominated by the convergent series
 \[
  \sum_{n=0}^\infty \frac{1}{(\alpha+1)(n+1)}\left(\frac{1}{4(n+1)^2}-\frac{1}{2^{2n+4}}\right)^{ \Large \nicefrac{\alpha}{2}}.
 \]
\end{remark}

\begin{definition}
 Let $E$ be an $\alpha$-Caccioppoli set in $\Omega$. We say that $E$ is a \emph{local $\alpha$-perimeter minimizer in $\Omega$} if in all bounded open sets $B\subseteq\Omega$ we have
 \[
  \Per(E;B)\leq\Per(F;B) \quad \text{for all $F$ such that $F\bigtriangleup E \subset\subset B$.}
 \]
\end{definition}

\subsection{Under weakened conditions}
The following definitions and results are analogous to the observations in \cite[sec. 1]{dPP}. We only prove one proposition, which was not used in \cite{dPP}.

\begin{definition}
 Let $E$ be an $\alpha$-Caccioppoli set in $\Omega$. We say that $E$ is a \emph{local $\alpha$-perimeter sub-minimizer in $\Omega$} if in all bounded open sets $B\subseteq\Omega$ we have
 \[
  \Per(E;B)\leq\Per(F;B) \quad \text{for all $F\subseteq E$ such that $E\backslash F \subset\subset B$.}
 \]
\end{definition}
The connection with minimizers is given by
\begin{proposition}\label{prop:Lew:BedMin}
  $E$ is a local $\alpha$-perimeter minimizer in $\Omega$ if and only if $E$ as well as $\Omega\backslash E$ is a local $\alpha$-perimeter sub-minimizer in $\Omega$.
\end{proposition}
The lower semicontinuity of the $\alpha$-perimeter implies
\begin{proposition}\label{prop:Lew:konv}
  Let $\{E_k\}_{k\in\mathbb{N}}$ and $E$ be $\alpha$-Caccioppoli sets in $\Omega$ with $E_k\subseteq E$ and suppose that $E_k$ locally converge to $E$ in $\Omega$. If all $E_k$'s are local $\alpha$-perimeter sub-minimizers in $\Omega$, then  $E$ is a local $\alpha$-perimeter sub-minimizer in $\Omega$ as well.
\end{proposition}
Furthermore, the existence of a so called sub-calibration ensures the sub-minimality.

\begin{definition}
  Let $E\subseteq\Omega$ be an $\alpha$-Caccioppoli set in $\Omega$ with $\partial E\cap \Omega\in C^2$. We call a vector field $\xi\in C^1(\Omega,\R^{m+1})$ an \emph{$\alpha$-sub-calibration of $E$ in $\Omega$} if it  fulfills
 \begin{enumerate}[\qquad (i)]
  \item $|\xi(z)|\leq y^\alpha$\neghphantom{$|\xi(z)|\leq y^\alpha$}\hphantom{ $\xi(z)=y^\alpha\cdot \nu_E(z)$} ~ for all $z\in \Omega$,
  \item $\xi(z)=y^\alpha\cdot \nu_E(z)$ ~ for all $z\in\partial E\cap \Omega$,
  \item $\operatorname{div} \xi(z)\leq 0$\neghphantom{$\operatorname{div} \xi(z)\leq 0$}\hphantom{$\xi(z)=y^\alpha\cdot \nu_E(z)$}~ for all $z\in\Omega$,
 \end{enumerate}
 where $\nu_E$ denotes the exterior unit normal vector field on $\partial E$.\footnote{Note that, in contrast to \cite{Davini}, our vector field has been weighted.}
\end{definition}
\begin{proposition}\label{prop:Lew:subKal}
 If $\xi$ is an $\alpha$-sub-calibration of $E$ in an open set $\EuScript{O}\subseteq\Omega$, then $E$ is a local $\alpha$-perimeter sub-minimizer in all $\Omega$.
\end{proposition}
Note that it suffices to find a sub-calibration on a subset of $\Omega$ which contains $E$ since we only deal with inner deformations. Finally, we add
\begin{proposition}\label{prop:Lew:ganz}
 If the cone $\mathcal{C}^\alpha_m$ is a local $\alpha$-perimeter sub-minimizer in $\R^m\times\R_{>0}\backslash\{x=0\}$, then $\mathcal{C}^\alpha_m$ is also a local $\alpha$-perimeter sub-minimizer in the whole  $\R^m\times\R_{\geq0}$.
\end{proposition}
\begin{proof}
 Firstly, we have for a bounded open set $\widetilde{B}\subset\R^m\times\R_{\geq0}$:
 \[
  \Per(\mathcal{C}^\alpha_m;\widetilde{B})\leq \Per(F;\widetilde{B})
 \]
for all $F\subseteq\mathcal{C}^\alpha_m$ such that $\mathcal{C}^\alpha_m\backslash F \ \subset\subset \  \widetilde{B}\backslash\{\,x=0\,\vee\,y=0\,\}$. Let now be $\widetilde{F}\subseteq \mathcal{C}^\alpha_m$ with $\mathcal{C}^\alpha_m\backslash\widetilde{F}\subset\subset\widetilde{B}$. For $\varepsilon>0$ we consider the set
\[
  \widetilde{F}_\varepsilon\coloneqq \widetilde{F}\cup \big( \mathcal{C}^\alpha_m \cap \{\,|x|<\varepsilon\,\ \vee\,y<\varepsilon\,\} \big).
\]
Hence,
\[
  \mathcal{C}^\alpha_m\backslash\widetilde{F}_\varepsilon \ \subset\subset \ \widetilde{B}\backslash\{\,x=0\,\vee\,y=0\,\},
\]
thus with the preliminary observation we have
\begin{align*}
 \Per(\mathcal{C}^\alpha_m;\widetilde{B})&\leq\Per(\widetilde{F}_\varepsilon;\widetilde{B}) \\
 &\leq \Per(\widetilde{F};\widetilde{B}) + c_1(m,\alpha,\widetilde{B})\cdot\{\varepsilon^{m+\alpha}+\varepsilon^\alpha+\varepsilon^{m-1}\} \\ &\xrightarrow{\varepsilon\searrow0}\Per(\widetilde{F};\widetilde{B}).
\end{align*}
\end{proof}
\section{First proof of theorem \ref{Satz:Lew:1}}
Arguing in this section as in \cite{dPP} we give a first proof of theorem \ref{Satz:Lew:1}. Unfortunately, this does not lead to our best bounds, but gives the $\alpha_m$'s as constructible numbers. This study is based on the analysis of the cubic polynomial
\[
 \mathbf{Q}_{m,\alpha}(t)\coloneqq (m-1)^4t^3-3(m-1)^2\alpha t^2-3(m-1)\alpha^2t+\alpha^4.
\]
\begin{lemma}\label{Lemma:Lew:Q}
For all $\displaystyle\alpha\geq\frac{2m^{\nicefrac{3}{2}}+3m-1}{(m-1)^2}$, we have
 \[ \mathbf{Q}_{m,\alpha}(t)\geq0 \quad \text{for all $t\geq0$.}\]
\end{lemma}
\begin{proof}
 For all admissible $m\in\{\,2,3,\ldots\,\}$ and $\alpha>0$, the polynomial $\mathbf{Q}_{m,\alpha}(-t)$  has one sign change in the sequence of its coefficients
 \[-(m-1)^4, \quad -3(m-1)^2\alpha, \quad 3(m-1)\alpha^2,\quad \alpha^4.\]
 Hence, due to Descartes' rule of signs, $\mathbf{Q}_{m,\alpha}$ always has one negative root. On the other hand, $\mathbf{Q}_{m,\alpha}$ has none, a double or two distinct positive roots.

The number of real roots of the cubic polynomial $\mathbf{Q}_{m,\alpha}$ is determined by its discriminant
\[
 \vartheta=-27(m-1)^6\alpha^6\cdot\{(m-1)^2\alpha^2-(6m-2)\alpha+1-4m\}.
\]
Summarizing, we have:
\begin{enumerate}[i)]
 \item If $\vartheta>0$, then $\mathbf{Q}_{m,\alpha}$ has one negative and two distinct positive roots.
 \item If $\vartheta=0$, then $\mathbf{Q}_{m,\alpha}$ has one negative and a double positive root.
 \item If $\vartheta<0$, then $\mathbf{Q}_{m,\alpha}$ has  only one negative root.
\end{enumerate}
The statement of the lemma then follows since $-\vartheta$ has the same sign as the quadratic polynomial
\[
 \mathbf{q}_m(\alpha)\coloneqq(m-1)^2\alpha^2-(6m-2)\alpha+1-4m,
\]
whose sole positive root is
\[\alpha=\frac{2m^{\nicefrac{3}{2}}+3m-1}{(m-1)^2}.\qedhere\]
\end{proof}
\begin{proof}[Proof of theorem \ref{Satz:Lew:1} (with concrete bounds)]
 We consider over $\R^m\times\R_{\geq0}$ the function
 \[
  F_{m,\alpha}(z)\coloneqq \frac14\left\{\alpha^2|x|^4-(m-1)^2 y^4\right\}.
 \]
It is \[ \nabla F_{m,\alpha}(z)= \big(\alpha^2|x|^2x \ \ , \ -(m-1)^2 y^3 \big).\]
Moreover, on $\{\,|\nabla F_{m,\alpha}|\neq0\,\}$ we have:
\begin{align*}
 \nabla_{x}\ \frac{|x|^2}{|\nabla F_{m,\alpha}|}&=\frac{2x}{|\nabla F_{m,\alpha}|}- \frac{3\alpha^4|x|^6x}{|\nabla F_{m,\alpha}|^3} \\
 \shortintertext{and}
 \frac{\partial}{\partial y}\ \frac{1}{|\nabla F_{m,\alpha}|}&=-\frac{3(m-1)^4 y^5}{|\nabla F_{m,\alpha}|^3}.
\end{align*}
Hence, \begin{align*}
 \operatorname{div}\left(-y^\alpha\vphantom{\frac{1}{1}}\frac{\nabla F_{m,\alpha}}{|\nabla F_{m,\alpha}|}\right)& = -\skalarProd{\nabla_{x} \ }{\ \frac{y^\alpha\alpha^2|x|^2x}{|\nabla F_{m,\alpha}|}}+\frac{\partial}{\partial y}\ \frac{(m-1)^2y^{\alpha+3}}{|\nabla F_{m,\alpha}|} \\[2ex]
 &\hspace{-5.5em}=-m\frac{y^\alpha\alpha^2|x|^2}{|\nabla F_{m,\alpha}|}- \skalarProd{x}{ \nabla_{x}\ \frac{y^\alpha\alpha^2|x|^2}{|\nabla F_{m,\alpha}|}}\\[0.5ex]
 &\hspace{-5.5em} \hspace*{1em}+\frac{(m-1)^2(\alpha+3)y^{\alpha+2}}{|\nabla F_{m,\alpha}|}  + (m-1)^2y^{\alpha+3}\frac{\partial}{\partial y}\ \frac{1}{|\nabla F_{m,\alpha}|}\\[2ex]
 &\hspace{-5.5em}= |\nabla F_{m,\alpha}|^{-3}\{-(m-1)\alpha^6y^\alpha|x|^8 - (m-1)^4(m+2)\alpha^2y^{\alpha+6} |x|^2 \\[0.5ex]
 &\hspace*{2em} + (m-1)^2(\alpha+3)\alpha^4y^{\alpha+2}|x|^6+(m-1)^6 \alpha y^{\alpha+8} \} \\[2ex]
 &\hspace{-5.5em}= - |\nabla F_{m,\alpha}|^{-3}(m-1)\alpha y^\alpha|x|^{6}\,\mathbf{Q}_{m,\alpha}\left(\frac{y^2}{|x|^2}\right)\{\alpha|x|^2-(m-1)y^2\}.
\end{align*}
For $k\in\mathbb{N}$ consider the sets
\[
 E_k\coloneqq\left\{\, z\in\R^m\times\R_{\geq0}: F_{m,\alpha}(z)\geq\frac1k  \,\right\}\subset\mathcal{C}^\alpha_m.
\]
They all are $\alpha$-Caccioppoli sets in $\R^m\times\R_{>0}\backslash\{x=0\}$ since
$$F_{m,\alpha}\in C^2(\big(\R^m\times\R_{>0}\backslash\{\,z=0\,\}\big)\backslash\mathcal{M}^\alpha_m),$$ whereby $\mathcal{M}^\alpha_m=\partial\mathcal{C}^\alpha_m=\{\,F_{m,\alpha}=0\,\}$. Furthermore, the $E_k$'s locally converge to $\mathcal{C}^\alpha_m=\{\, F_{m,\alpha}\geq0 \,\}$.

With lemma \ref{Lemma:Lew:Q} we have
\[
 \mathbf{Q}_{m,\alpha}\left(\frac{y^2}{|x|^2}\right)\geq0 \quad  \text{for all \ $x\neq0$, \ $y\geq0$}, \text{ and for all } \alpha\geq\frac{2m^{\nicefrac32}+3m-1}{(m-1)^2}
\]
consequently, due to the above computation of the divergence, the vector filed
\[
 \xi_{+}(z)\coloneqq -y^\alpha\frac{\nabla F_{m,\alpha}(z)}{|\nabla F_{m,\alpha}(z)|}
\]
is an $\alpha$-sub-calibration for each $E_k$ in $\{0<\sqrt{m-1}y<\sqrt{\alpha}|x|\}$.

Hence, propositions \ref{prop:Lew:subKal}, \ref{prop:Lew:konv} and \ref{prop:Lew:ganz} ensure that $\mathcal{C}^\alpha_m$ is a local $\alpha$-perimeter sub-minimizer in the whole $\R^m\times\R_{\geq0}$.

In view of the characterization of $\alpha$-perimeter minimizing sets, cf. proposition \ref{prop:Lew:BedMin}, the claim of theorem \ref{Satz:Lew:1} follows for
\[
\alpha\geq\frac{2m^{\nicefrac32}+3m-1}{(m-1)^2},
\]
after proving the sub-minimality of the complement of $\mathcal{C}^\alpha_m$. We therefor argue as above considering the sets
\[
 D_k\coloneqq\left\{\, z\in\R^m\times\R_{\geq0}: F_{m,\alpha}(z)\leq-\frac1k  \,\right\}
\]
and the vector field
\[
 \xi_{-}(z)\coloneqq y^\alpha\frac{\nabla F_{m,\alpha}(z)}{|\nabla F_{m,\alpha}(z)|} \quad\text{on ~ $\{F_{m,\alpha}<0\}$.}\qedhere
\]
\end{proof}

\begin{remark}
 All previous computations were carried out by hand.
\end{remark}

\begin{remark}
 For $m\geq14$ we have $\frac{2m^{\nicefrac32}+3m-1}{(m-1)^2}>\frac{12}{m}$ and $\frac{12}{m}$ is an upper bound for our best $\alpha_m$'s.
\end{remark}

\begin{remark}\label{Lew:Bed}
 Improvements of these bounds can be achieved by an alternative auxiliary function. As seen in the proof, such a function $F$ should fulfill the following conditions
 \begin{enumerate}[1.]
 \item $F\in C^2(\big(\R^m\times\R_{>0}\backslash\{\,x=0\,\}\big)\backslash\mathcal{M}^\alpha_m)\cap C^0(\R^m\times\R_{\geq0}) $,
 \item $\{\,F\geq0\,\}=\mathcal{C}^\alpha_m$, \quad $\{\,F=0\,\}=\partial \mathcal{C}^\alpha_m=\mathcal{M}^\alpha_m$,
 \item $\displaystyle F\cdot\operatorname{div}\left(-y^\alpha\frac{\nabla F}{|\nabla F|}\right)\leq0$  ~ in $\{\,\nabla F \neq0\,\}$.
 \end{enumerate}
\end{remark}

\begin{remark}\label{rem:minimizingArea}
In fact, corresponding auxiliary functions can be found in papers concerning the minimizing property of Lawson's cones
, namely
\begin{itemize}
 \item in \cite{MM}: $F(x,y)=\big(|x|^2-|y|^2\big)\big(|x|^2+|y|^2\big)$, for $k=h=4$.
 \item in \cite{CM}: $$F(x,y)=\big((h-1)|x|^2-(k-1)|y|^2\big)\big((5k-h-4)(h-1)|x|^2-(5h-k-4)(k-1)|y|^2\big),$$
 for $k+4<5h$ and $(k,h)\neq(3,5)$, and for $h+4<5k$ and $(k,h)\neq(5,3)$.
 \item in \cite{BM}: $$F(x,y)= \big((h-1)|x|^2-(k-1)|y|^2\big)\cdot\begin{cases}\big((h-1)|x|^2\big)^\beta, \text{ in ``$\{F>0\}$'',}\\\big((k-1)|y|^2\big)^\beta, \text{ in ``$\{F<0\}$'',} \end{cases}$$
 where $\beta$ was chosen in a way, that such an argumentation was admissible for all Lawson's cones.
 \item in \cite{dPP}: $F(x,y)=\frac14\big(|x|^2-|y|^2\big)\big(|x|^2+|y|^2\big)$, for $k=h\geq4$.
\end{itemize}
Note that
\begin{itemize}
 \item
in \cite{CM,BM} computer algebra systems were used to perform the symbolic manipulations.
\item
the argumentation using sub-calibration method from \cite{dPP} is applicable to the function
\[F(x,y)=\frac{1}{4}\big((h-1)|x|^2-(k-1)|y|^2\big)\big((h-1)|x|^2+(k-1)|y|^2\big) \]
and yields the minimality of all Lawson's cones with
\begin{align*}
 (k,h)\notin\{&(2,7),(2,8),(2,9),(2,10),(2,11),(3,5),\\ &(5,3),(7,2),(8,2),(9,2),(10,2),(11,2)\}.
\end{align*}
However, we have already performed such computations above and the exceptional cases correspond to the given bounds in lemma \ref{Lemma:Lew:Q} for integer values, where $k$ and $h$ take over the parts of $m$ and $\alpha+1$.
\end{itemize}
\end{remark}

\begin{remark}\label{rem:Davini}
 With the aid of a suitable parametrization \textsc{Davini} detected the existence of an auxiliary function which was applicable to all Lawson's cones. All his computations he carried out by hand, cf. \cite{Davini}.
\end{remark}

\section{Second proof of theorem \ref{Satz:Lew:1} with better bounds}
Since the hypercones $\mathcal{M}^\alpha_m=\partial\mathcal{C}^\alpha_m$ are invariant under the action of $SO(m)$ on the first $m$ components, we will look for a foliation consisting of extremal hypersurfaces with the same type of symmetry. In fact, recalling \eqref{eq:Lew:varprob}, a dimension reduction and the special parametrization\footnote{Note that the simplification in \cite{Davini} towards the argumentation as in \cite{BdGG} comes from such a parametrization.}
\begin{equation}\label{eq:Lew:parametKurve}
 \begin{cases}
  |x| &= \mathrm{e}^{v(t)}\cdot\cos t,\\
  \hphantom{|}y &=\mathrm{e}^{v(t)}\cdot\sin t,
 \end{cases}
\end{equation}
with  $v\in C^2(0,\frac\pi2)$ yields as Euler-Lagrange equation
\begin{equation}\label{eq:Lew:DiffglZ}
 \ddot{v}=\Big(1+\dot{v}^2\Big)\cdot\left\{m+\alpha+\frac{m-\alpha-1-(m+\alpha-1)\cos(2t)}{\sin(2t)}\cdot \dot{v}\right\},
\end{equation}
cf. \cite{Davini}, where $m$ and $\alpha$ take over the parts of $k$ and $h-1$.

Hence, with $w\coloneqq \dot{v}$ the initial problem reduces to a question about the behavior of solutions of the following ordinary differential equation of first order:
\begin{equation}\label{eq:Lew:DiffglY}
 \dot{w}=\big(1+w^2\big)\cdot\left\{m+\alpha+\frac{m-\alpha-1-(m+\alpha-1)\cos(2t)}{\sin(2t)}\cdot w\right\}.
\end{equation}
The existence of a solution follows, for example, from the existence of an upper and a lower solution of \eqref{eq:Lew:DiffglY}. Arguing as \textsc{Davini} we will directly give an upper solution and the difficult part is in finding the conditions on $m$ and $\alpha$ under which a suitable lower solution exists. Note that we push the argumentation from \cite{Davini} to the extreme, since $\alpha>0$ is real valued and not necessarily an integer. Our study is based on the analysis of the quartic polynomial

\begin{gather*}
\P(\gamma)\coloneqq a_4\gamma^4+a_3\gamma^3+a_2\gamma^2+a_1\gamma+a_0,
\shortintertext{with}
\begin{aligned}
a_4&=(m+\alpha)^3,\\[1ex]
 a_3&=-(m+\alpha)^2(m+\alpha+1),\\[1ex]
 a_2&=(m+\alpha)(2m+6\alpha-4m\alpha-1),\\[1ex]
 a_1&=4m^2\alpha+4\alpha^2m-4\alpha^2-5\alpha-m+1,\\[1ex]
 a_0&=-8(m-1)\alpha.
\end{aligned}
\end{gather*}

\begin{lemma}\label{Lemma:Lew:biquadr}
 There exists an algebraic number $\alpha_m >\frac2m$ such that for all $\alpha\geq\alpha_m$ we can find a value $\gam\in(0,1-\frac{1}{m+\alpha})$ with \[\P(\gam)\geq0.\]
\end{lemma}

\begin{proof}
Note that
 \begin{align*}
 \P(0)&=-8(m-1)\alpha<0\\ \shortintertext{and}  \P(1-\textstyle\frac{1}{m+\alpha})&=-\frac{8(m-1)\alpha}{m+\alpha}<0.
\end{align*}
Further, for all admissible $m\in\{\,2,3,\ldots\,\}$ and $\textstyle\alpha>\frac2m$ the coefficients of  $\P$ fulfill:
\begin{align*}
 a_4 &=(m+\alpha)^3>0,\\[1ex]
 a_3 &=-(m+\alpha)^2(m+\alpha+1)<0,\\[1ex]
 a_1 &=5\alpha(\textstyle\frac{m^2}{4}-1)+4\alpha^2(m-1)+m(\frac{11}{4}m\alpha-1)+1>0,\\[1ex]
 a_0 &=-8(m-1)\alpha<0,
\end{align*}
consequently, $\P(-\gamma)$ has, regardless of the value $a_2$, always one sign change in the sequence of its coefficients $ a_4,\ -a_3,\ a_2,\ -a_1,\ a_0$. Hence, due to Descartes' rule of signs,  $\P$ always has one negative root. Moreover, we have
\begin{align*}
&\P(\gamma+1-\textstyle\frac{1}{m+\alpha})= \widetilde{a}_4\gamma^4+\widetilde{a}_3\gamma^3+\widetilde{a}_2\gamma^2+\widetilde{a}_1\gamma+\widetilde{a}_0,\\[1.5ex]
&\quad\begin{aligned}
 \text{with} \quad &\widetilde{a}_4=(m+\alpha)^3>0,\\[1ex]
 &\widetilde{a}_3=(m+\alpha)^2(3m+3\alpha-5)>0,\\[1ex]
 &\widetilde{a}_2=(m+\alpha)\{(m-2)(3m-4)+\textstyle\frac{2\alpha}{m}(m^2-3m+\frac32\alpha m)\}>0,\\[1ex]
 &\widetilde{a}_0=-\textstyle\frac{8(m-1)\alpha}{m+\alpha}<0,
 \end{aligned}
\end{align*}
thus, regardless of the value $\widetilde{a}_1$, we always have one sign change in the sequence of coefficients of the polynomial $\P(\gamma+1-\textstyle\frac{1}{m+\alpha})$. In other words, $\P$ always has one root in $(1-\frac{1}{m+\alpha},\infty)$.

All in all, $\P$ has none, a double or two distinct roots in the interval $(0,1-\frac{1}{m+\alpha})$. To determine the nature of roots of the quartic equation
 \begin{equation}\label{eq:Lew:biquadr}
  \P(\gamma)=0.
 \end{equation}
we convert it by the change of variable $\gamma=u+\frac{m+\alpha+1}{4(m+\alpha)}$ to the depressed quartic
 \begin{equation}\label{eq:Lew:reduziert}
   u^4+pu^2+qu+r=0, \tag{\ref{eq:Lew:biquadr}*}
 \end{equation}
 with coefficients
  \begin{align*}
  p &= \textstyle -\frac{1}{8(m+\alpha)^2}\{3m^2-10m+11+3\alpha^2+2(19m-21)\alpha\}<0 ,\\[1.7ex]
  q &= \textstyle -\frac{1}{8(m+\alpha)^3}\{ \alpha^3+\alpha^2(11-13m)-\alpha(m-1)(13m+23)+(m-3)(m-1)^2\}, \\[1.7ex]
  r &= \textstyle -\frac{1}{256(m+\alpha)^4} \begin{aligned}[t]\{3 \alpha ^4+172 \alpha ^3-1630 \alpha ^2+204 \alpha +3 m^4-180 \alpha  m^3-20 m^3-366\alpha ^2 m^2 \\[1ex]+1796 \alpha  m^2+34 m^2-180 \alpha ^3 m+1988 \alpha ^2 m-1788 \alpha  m+12 m-45 \}, \end{aligned}
 \end{align*}
and consider its resolvent cubic, namely
 \begin{equation}\label{eq:Lew:kubRes}
 \zeta^3+2p\zeta^2+(p^2-4r)\zeta-q^2=0 \tag{\ref{eq:Lew:reduziert}*}.
\end{equation}
We have $p<0$ and $p^2-4r>0$ as
\begin{align*}
16(m+\alpha)^4(p^2-4r)=&3 \alpha ^4+4 (3m-5)\alpha ^3+\left(274 m^2-316 m+50\right)\alpha ^2 \\[1ex]
& \quad+4(m-1) (3 m^2+52m+45)\alpha+(m-1)^2 (3 m^2-14m+19).
\end{align*}
Consequently, \eqref{eq:Lew:kubRes} has no negative roots, since there is no sign change in the sequence of the coefficients $-1,\ 2p,\ 4r-p^2,\ -q^2$. On the other hand, \eqref{eq:Lew:kubRes} has one or three positive roots depending on the sign of its discriminant
\[
 \theta=4p^2(p^2-4r)^2-4(p^2-4r)^3-36p(p^2-4r)q^2+32p^3q^2-27q^4.
\]
In view of the foregoing, it follows:
\begin{enumerate}[i)]
 \item If $\theta>0$, then $\P$ has two distinct roots in $(0,1-\frac{1}{m+\alpha})$.
 \item If $\theta=0$, then $\P$ has one double root in $(0,1-\frac{1}{m+\alpha})$.
 \item If $\theta<0$, then $\P$ has no roots in $(0,1-\frac{1}{m+\alpha})$.
 \end{enumerate}
So, the statement of the lemma follows for such values of $m$ and $\alpha$ for which $\theta=\theta_m(\alpha)\geq0$. We have\label{Lew:Polynom}:
\begin{align*}
\frac{(m+\alpha)^{12}}{16\alpha(m-1)}\cdot\theta_m(\alpha)=&\\
&\hspace{-5.7em}  16(m-1)^2\alpha^8 \\[1ex]
&\hspace{-5.7em} - 4(m-1)(8m^2+3)\alpha^7 \\[1ex]
&\hspace{-5.7em} - (16 m^4-256 m^3+584 m^2-496 m+153)\alpha^6 \\[1ex]
&\hspace{-5.7em} + 2 (32 m^5-224 m^4+1238 m^3-2738 m^2+2545 m-852)\alpha^5 \\[1ex]
&\hspace{-5.7em} - (m-1) (16 m^5+48 m^4-1712 m^3+6672 m^2-4321 m-641)\alpha^4\\[1ex]
&\hspace{-5.7em} - 2 (16 m^7-208 m^6+250 m^5+2302 m^4-3214 m^3-588 m^2+1566 m-123)\alpha^3\\[1ex]
&\hspace{-5.7em} + (16 m^8-192 m^7+984 m^6-2864 m^5+1001 m^4+4184 m^3-3870 m^2+794 m-52)\alpha^2\\[1ex]
&\hspace{-5.7em} - 2 (m-1) (22 m^6-148 m^5+363 m^4-381 m^3+185 m^2-60 m+2)\alpha\\[1ex]
&\hspace{-5.7em} - (m-2)^3 (m-1)^2 m \eqqcolon \p(\alpha).
\end{align*}
Note that the polynomial $\p$ has three changes of sign in its sequence of coefficients if $m=2,\ldots,6$ and five changes if $m\geq7$, so that Descartes' rule of signs is not applicable to show that $\p$ has only one positive root. To prove the latter we will now apply Sturm's theorem. For that purpose
we consider the canonical Sturm chain
\[
\p_{,0}(\alpha),  \ \p_{,1}(\alpha), \ldots, \ \p_{,8}(\alpha)
\]
and count the number of sign changes in these sequences for $\alpha=0$ and $\alpha\to\infty$:

$$
\begin{array}{ll"c:c|}
 &  &  \alpha= 0 & \alpha\to\infty \\[0.5ex] \noalign{\hrule height 1pt}
 \rule{0pt}{25pt} \multirow{6}{4ex}{\rotatebox[origin=c]{90}{\text{sign of}}}
 & \p_{,0}(\alpha) & \text{\renewcommand{\arraystretch}{1.4}\footnotesize$\begin{array}{@{}cl@{}} 0 & m=2 \\ - & m\geq3 \end{array}$} & + \\[3.5ex] \cdashline{2-4}
 \rule{0pt}{20pt} & \p_{,1}(\alpha) & - & + \\[2ex] \cdashline{2-4}
 \rule{0pt}{20pt} & \p_{,2}(\alpha) & + & + \\[2ex] \cdashline{2-4}
 \rule{0pt}{25pt} & \p_{,3}(\alpha) & + &  \rule{10pt}{0pt}\text{\renewcommand{\arraystretch}{1.4}\footnotesize$\begin{array}{@{}cl@{}} - & m=2,\ldots,28 \\ + & m\geq29 \end{array}$}\rule{10pt}{0pt} \\[3.5ex] \cdashline{2-4}
 \rule{0pt}{25pt} & \p_{,4}(\alpha) & \text{\renewcommand{\arraystretch}{1.4}\footnotesize$\begin{array}{@{}cl@{}} - & m=2 \\ + & m\geq3 \end{array}$} & - \\[3.5ex] \cdashline{2-4}
 \rule{0pt}{35pt} & \p_{,5}(\alpha) & \text{\renewcommand{\arraystretch}{1.3}\footnotesize$\begin{array}{@{}cl@{}} - & m=2,3 \\ + & m=4,5 \\ - & m\geq 6 \end{array}$} & \text{\renewcommand{\arraystretch}{1.3}\footnotesize$\begin{array}{@{}cl@{}} - & m=2,\ldots,4 \\ + & m=5,\ldots,10 \\ - & m\geq11 \end{array}$} \\[5ex] \cdashline{2-4}
 \rule{0pt}{25pt} & \p_{,6}(\alpha) & \text{\renewcommand{\arraystretch}{1.4}\footnotesize$\begin{array}{@{}cl@{}} + & m=2 \\ - & m\geq3 \end{array}$} & \text{\renewcommand{\arraystretch}{1.4}\footnotesize$\begin{array}{@{}cl@{}} + & m=2,\ldots,22 \\ - & m\geq23 \end{array}$}  \\[3.5ex] \cdashline{2-4}
 \rule{0pt}{25pt} & \p_{,7}(\alpha) & \rule{10pt}{0pt}\text{\renewcommand{\arraystretch}{1.4}\footnotesize$\begin{array}{@{}cl@{}} + & m=2,\ldots,6 \\ - & m\geq7  \end{array}$}\rule{10pt}{0pt} & + \\[3.5ex] \cdashline{2-4}
 \rule{0pt}{20pt} & \p_{,8}(\alpha) & + & + \\[2ex]
 \noalign{\hrule height 1pt}
 \multicolumn{2}{c"}{\text{sign changes}}  & \rule{0pt}{20pt} 3 & 2 \\[1.5ex]\hline
\end{array}
$$
Hence, due to Sturm's theorem, the polynomial $\p$ has always $3-2=1$ positive root which we denote by $\alpha_m$. Moreover we have
\begin{align*}
 m^8\cdot \p\left(\frac2m\right)= {} & -25 m^{14}-80 m^{13}+1611 m^{12}-5114 m^{11}-2544 m^{10}-19620 m^9\\[-0.2ex]
 & +65904 m^8 -135888m^7 +228832 m^6 -215760 m^5+111152 m^4 \\[1.2ex]
 & -18688 m^3 -7232 m^2-6656 m+4096 <0
 \shortintertext{and}
 m^8\cdot\p\left(\frac{12}{m}\right)= {} & 1775 m^{14}-23560 m^{13}+74111 m^{12}+324326 m^{11}-1065244 m^{10}\\[-0.2ex]
 &-8010880m^9 +62969424 m^8-283180848 m^7+790863552 m^6\\[1.2ex]
 & -674075520 m^5-1637169408 m^4+2203656192 m^3+5992869888m^2 \\[1.2ex]
 &-13329432576 m+6879707136  > 0 \quad \text{for all $m\geq2$},
\end{align*}
thus,
\[
\frac2m<\alpha_m<\frac{12}{m} \ {}.\qedhere
\]
\end{proof}
\begin{remark}
The lengthy symbolic manipulations were completed here with the aid of the \emph{Wolfram Language} on a \emph{Raspberry Pi 2, Model B}. The following computations will again be carried out by hand:
\end{remark}
\begin{proof}[Proof of theorem \ref{Satz:Lew:1}]
Denoting the right-hand side of \eqref{eq:Lew:DiffglY} by $\Hdif(t,w)$ we see that
 \[
 \g(t)\coloneqq(m+\alpha)\cdot\frac{\sin(2t)}{(m+\alpha-1)\cos(2t)-(m-\alpha-1)}
\]
fulfills
\[\Hdif(t,\g(t))=0 \quad \text{on $(0,\t)\cup(\t,\textstyle\frac\pi2)$,}\]
where
\[
 \t\coloneqq\frac12\arccos\left(\frac{m-\alpha-1}{m+\alpha-1}\right) = \arctan\sqrt{\frac{\alpha}{m-1}}.
\]
Since $\g'(t)\geq0$, the function $\g$ is an upper solution of \eqref{eq:Lew:DiffglY}. As we are interested in a solution of \eqref{eq:Lew:DiffglY}, which has the same growth properties as $\g$, it is natural to ask for a lower solution of the form $\gamma\cdot\g$ with $ \gamma\in(0,1)$, i.e., we should have
\begin{equation}\label{eq:Lew:Unterfkt}
 \gamma\cdot \g'(t) \leq \Hdif(t,\gamma\cdot \g(t)) \qquad \text{for all $t\in (0,\t)\cup(\t,\textstyle\frac\pi2)$.}
\end{equation}
For $t\neq\t$ this is equivalent to
\begin{align}\label{eq:Lew:Unterfkt**}
&\qquad a\cdot\cos^2(2t)-2b\cdot\cos(2t)+c\geq0,\tag{\ref{eq:Lew:Unterfkt}*}
\shortintertext{with}
 a&= (1-\gamma)\big((m+\alpha-1)^2-\gamma^2(m+\alpha)^2 \big),\notag\\[1ex]
 b&= (m-\alpha-1)\big(m+\alpha-1 -\gamma(m+\alpha) \big),\notag\\[1ex]
 c&= (1-\gamma)\gamma^2(m+\alpha)^2-2\gamma(m+\alpha-1)+(1-\gamma)(m-\alpha-1)^2.\notag
\end{align}
Note that \eqref{eq:Lew:Unterfkt**} is valid on $(0,\frac\pi2)$ as long as $\gamma\in(0,1-\frac{1}{m+\alpha})$. The latter is equivalent to $a>0$. Hence, the left hand side of  \eqref{eq:Lew:Unterfkt**} is bounded below by
 $$ c-\frac{b^2}{a}.$$
 In other words, to find an adequate lower solution, it suffices to find conditions on $m$ and $\alpha$ under which a $\gamma\in(0,1-\frac{1}{m+\alpha})$ exists with
\begin{align*}
 c-\frac{b^2}{a}&\geq0 \\[2ex]
  \hspace{-8em}\underset{m\geq2, \ \alpha>\frac2m}{\overset{\gamma\in(0,1)}{\Longleftrightarrow}}\quad \P(\gamma) &\geq0,
\end{align*}
and lemma \ref{Lemma:Lew:biquadr} yields the desired conclusion. Consequently, we gain for $\gam$:
\begin{equation*}
 \gam\cdot \g'(t) \leq \Hdif(t,\gam\cdot \g(t)) \qquad \text{on $(0,\t)\cup(\t,\textstyle\frac\pi2)$,}
\end{equation*}
i.e., the function $\gam\cdot \g$ is a lower solution of \eqref{eq:Lew:DiffglY}, so that we can proceed as in \cite{Davini}: Due to results from classical ordinary differential equations theory it follows the existence of a $C^1$-solution $\w$ of \eqref{eq:Lew:DiffglY} on $(0,\t)\cup(\t,\textstyle\frac\pi2)$. Moreover, $\w$ satisfies
\begin{align*}
 &0 <\gam\cdot \g(t)\leq \w(t)\leq \g(t) \quad \text{on $(0,\t)$}
 \shortintertext{and}
 \quad & 0> \gam\cdot \g(t)\geq \w(t)\geq \g(t) \quad \text{on $(\t,\textstyle\frac\pi2)$,}
 \end{align*}
as well as
\begin{gather*}
 \lim_{t\nearrow\t} \w(t) = +\infty, \quad \lim_{t\searrow\t} \w(t) = -\infty,\\[2ex] \qquad \lim_{t\searrow0}\w(t)=0=\lim_{t\nearrow\frac{\pi}{2}}\w(t).
\end{gather*}
Let us denote by $\v$ the antiderivative of $\w$ with
$$\lim\limits_{t\searrow0}\v(t)=0 \quad\text{and}\quad \lim\limits_{t\nearrow\frac\pi2}\v(t)=0.$$
Reconstructing the auxiliary function from its level curves which are parametrized by
\begin{equation*}
 \begin{cases}
  |x| &= \lambda\cdot\mathrm{e}^{\v(t)}\cdot\cos t,\\
  \hphantom{|}y &=\lambda\cdot\mathrm{e}^{\v(t)}\cdot\sin t,
 \end{cases}
\end{equation*}
with $\lambda>0$ and $t\in(0,\t)\cup(\t,\frac\pi2)$, we gain
\begin{equation*}
 \F(x,y)\coloneqq\begin{cases}
                        &\sqrt{|x|^2 + y^2}\cdot\mathrm{e}^{-\v(\arctan\frac{y}{|x|})}, \quad 0<\arctan\frac{y}{|x|}<\t, \\[2ex]
                        -\hspace*{-3ex}&\sqrt{|x|^2 + y^2}\cdot\mathrm{e}^{-\v(\arctan\frac{y}{|x|})}, \quad \t<\arctan\frac{y}{|x|}<\frac{\pi}{2}.
                       \end{cases}
\end{equation*}
Note that, since $\v$ satisfies \eqref{eq:Lew:DiffglZ}, we obtain
\[
 \operatorname{div}\left(-y^\alpha\frac{\nabla \F}{|\nabla \F|}\right)=0, \quad \text{on $\big(\R^m\times\R_{>0}\backslash\{\,x=0\,\}\big)\backslash\mathcal{M}^\alpha_m$.}
\]
We than conclude as in our first proof above because $\F$ has the desired properties, cf. remark \ref{Lew:Bed}.
\end{proof}
\begin{remark}\label{Lew:Bem:Koeff}
The crucial ingredient in our argumentation was to find conditions on $m\geq2$ and $\alpha>0$ under which a $\gamma\in(0,1)$ exists such that \eqref{eq:Lew:Unterfkt**} is fulfilled on $(0,\t)\cup(\t,\frac{\pi}{2})$.
For $t\to\t$ the inequality \eqref{eq:Lew:Unterfkt**} is equivalent to
\[
(1-\gamma)\gamma\geq\frac{2(m+\alpha-1)}{(m+\alpha)^2}.
\]
The last inequality has solutions in $(0,1)$ as long as $m+\alpha\geq4+\sqrt{8}$.
Hence, $$\max\{\,4-m+\sqrt{8},0\,\}$$ are lower bounds for the optimal $\alpha_m$'s. With our values we have already reached the lower bounds quite close, so, for $m=4$ we have
\[
\alpha_4-\sqrt{8}<\frac{1}{1000}.
\]
\end{remark}

\begin{acknowledgement}
This paper is a part of my PhD thesis written under supervision of Prof. Ulrich \textsc{Dierkes}.
\end{acknowledgement}


\begin{thebibliography}{9}

\bibitem{BM}
\textsc{D.  Benarr\'{o}s and  M.  Miranda,}
\textit{Lawson cones and the Bernstein theorem,}
Advances in geometric analysis and continuum mechanics (1995), pp.~44\,--\,56.

\bibitem{BdGG}
\textsc{E.  Bombieri, E.  De  Giorgi and  E.  Giusti,}
\textit{Minimal cones and the Bernstein problem},
Invent. Math.~\textbf{7} (1969), pp.~243\,--\,268.

\bibitem{CM}
\textsc{P.  Concus and  M.  Miranda,}
\textit{MACSYMA and minimal surfaces,}
Proc. of Symposia in Pure Mathematics, by the Amer. Math. Soc.,~44, (1986), pp.~163\,--\,169.

\bibitem{Davini}
\textsc{A.  Davini,}
\textit{On calibrations for Lawson's cones,}
Rend. Sem. Mat. Univ. Padova~\textbf{111} (2004), pp.~55\,--\,70.

\bibitem{dPP}
\textsc{G.  de Philippis and  E. Paolini,}
\textit{A short proof of the minimality of Simons cone,}
Rend. Sem. Mat. Univ. Padova~\textbf{121} (2009), pp.~233\,--\,241.

\bibitem{Dierkes:erstErg}
\textsc{U.  Dierkes,}
\textit{Minimal hypercones and $C^{0,\nicefrac12}$-minimizers for a singular variational problem,}
 Indiana University Math. Journ.~\textbf{37} (1988), no. 4, pp.~841\,--\,863.

\bibitem{Dierkes:verbessErg}
\textsc{U.  Dierkes,}
\textit{A classification of minimal cones in $\R^n\times\R^+$ and a counterexample to interior regularity of energy minimizing functions,}
Manuscripta Math.~\textbf{63} (1989), pp.~173\,--\,192.

\bibitem{DHT}
\textsc{U.  Dierkes,  S.  Hildebrandt and  A.  J.  Tromba,}
 \textit{Global analysis of minimal surfaces,} Revised and enlarged second edition,
 Grundlehren der mathematischen Wissenschaften~341.
 Springer-Verlag Berlin Heidelberg, 2010.

\bibitem{Federer}
\textsc{H.  Federer,}
\textit{Geometric measure theory,}
 Grundlehren der mathematischen Wissenschaften~153.
 Springer-Verlag New York Inc., New York, 1969.

\bibitem{GH1}
\textsc{M.  Giaquinta and  S.  Hildebrandt,}
\textit{Calculus of variations 1,} The Lagrangian formalism, Corrected second printing,
 Grundlehren der mathematischen Wissenschaften~310.
 Springer-Verlag Berlin Heidelberg, 2004.

\bibitem{GH2}
\textsc{M.  Giaquinta and  S.  Hildebrandt,}
\textit{Calculus of variations 2,} The Lagrangian formalism, Corrected second printing,
 Grundlehren der mathematischen Wissenschaften~311.
 Springer-Verlag Berlin Heidelberg, 2004.

\bibitem{Giusti}
\textsc{E.  Giusti,}
\textit{Minimal surfaces and functions of bounded variation,}
Monographs in mathematics~80.
Birkhäuser Boston, Inc., 1984.

\bibitem{Lawlor}
\textsc{G. R. Lawlor,}
\textit{A sufficient criterion for a cone to be area-minimizing,}
Mem. Amer. Math. Soc. 91 (1991), no. 446.

\bibitem{Lawson}
\textsc{H.  B.  Lawson,  Jr.,}
\textit{The equivariant Plateau problem and interior regularity,}
Trans. Amer. Math. Soc.~\textbf{173} (1972), pp.~231\,--\,249.

\bibitem{Maggi}
\textsc{F.  Maggi,}
\textit{Sets of finite perimeter and geometric variational problems,} An introduction to geometric measure theory,
Cambridge studies in advanced mathematics~135.
Cambridge University Press, 2012.

\bibitem{MM}
\textsc{U.  Massari and  M.  Miranda,}
\textit{A remark on minimal cones,}
Boll. Un. Mat. Ital.~\textbf{6} (2-A) (1983), pp.~123\,--\,125.

\bibitem{Simoes}
\textsc{P.  A.  Q.  Sim\~{o}es,}
\textit{A class of minimal cones in $\R^n$, $n\geq8$, that minimize area,}
Ph. D. thesis, University of California, Berkley, CA 1973.

 \end{thebibliography}
\end{document}